\documentclass[preprint,12pt]{elsarticle}

\usepackage{mathrsfs}
\usepackage{amssymb}
\usepackage{amsthm}
\usepackage{amsmath}
\usepackage{dsfont}
\usepackage{graphicx}
\usepackage{float}

\newtheorem{theorem}{Theorem}[section]
\newtheorem{lemma}[theorem]{Lemma}
\newtheorem{corollary}[theorem]{Corollary}
\newtheorem{conj}[theorem]{Conjecture}
\newtheorem{problem}[theorem]{Problem}
\newtheorem{claim}{Claim}
\newtheorem{prop}{Proposition}

\allowdisplaybreaks

\journal{Linear and Multilinear Algebra}

\begin{document}

\begin{frontmatter}

\title{On trees with extremal extended spectral radius}

\author[1]{Junli Hu}
\ead{hujunli3020@163.com}

\author[1,2]{Xiaodan Chen\corref{cor1}}
\ead{x.d.chen@live.cn}
\cortext[cor1]{Corresponding author.}

\author[1]{Qiuyun Zhu}
\ead{zhuqiuyun970302@163.com}

\address[1]{College of Mathematics and Information Science, Guangxi University,\\
Nanning 530004, Guangxi, P. R. China}

\address[2]{Guangxi Center for Mathematical Research, Guangxi University,\\
Nanning 530004, Guangxi, P. R. China}

\begin{abstract}
Let $G$ be a simple connected graph with $n$ vertices, and let $d_{i}$ be the degree of the vertex $v_{i}$ in $G$.
The extended adjacency matrix of $G$ is defined so that the $ij$-entry is 
$\frac{1}{2}(\frac{d_{i}}{d_{j}}+\frac{d_{j}}{d_{i}})$ if the vertices $v_i$ and $v_j$ are adjacent in $G$, and 0 otherwise.
This matrix was originally introduced for developing novel topological indices used in the QSPR/QSAR studies.
In this paper, we consider extremal problems of the largest eigenvalue of the extended adjacency matrix
(also known as the extended spectral radius) of trees.
We show that among all trees of order $n\geq 5$,
the path $P_n$ (resp., the star $S_n$) uniquely minimizes (resp., maximizes) the extended spectral radius.
We also determine the first five trees with the maximal extended spectral radius.
\end{abstract}


\end{frontmatter}

\section{Introduction}

Let $G$ be a simple connected graph with vertex set $V(G)=\{v_1,v_2,\dots,v_n\}$ and edge set $E(G)$.
For $i\in\{1,2,\dots,n\}$, denote by $d_{i}$  the degree of the vertex $v_i$ in $G$.
We will write $\Delta(G)$ and $\delta(G)$ for the maximum degree and minimum degree of $G$, respectively.
We also denote by $G-v_i$ the graph obtained from $G$ by deleting the vertex $v_i$ (and those edges incident with it).
As usual, let $S_{n}$, $P_{n}$, and $C_{n}$ denote the star, path, and cycle on $n$ vertices, respectively.

The {\it extended adjacency matrix} of a graph $G$ is defined as the $n\times n$ matrix $A_{ex}(G)=(a_{ij}^{ex})$, where
$$
a_{ij}^{ex}=\left\{\begin{matrix}
\frac{1}{2}(\frac{d_{i}}{d_{j}} + \frac{d_{j}}{d_{i}})&\text{if } v_{i}v{_j}\in E(G),\\
0                                                            &\text{otherwise}.
\end{matrix}\right.
$$
This matrix was originally introduced by Yang et al. \cite{Yang}
for devising two new topological indices---the {\it extended spectral radius} and {\it extended energy},
which are defined as the largest eigenvalue and the sum of the absolute values of all eigenvalues of the matrix, respectively.
The two indices have turned out to be nice molecular descriptors that show lower degeneracy
and good performance in correlation with many physicochemical properties and biological activities of organic compounds.

Recall that the {\it (ordinary) adjacency matrix} of a graph $G$ is defined to be the $n\times n$ matrix $A(G)=(a_{ij})$, where
$$
a_{ij}=\left\{\begin{matrix}
1&\text{if } v_{i}v{_j}\in E(G),\\
0&\text{otherwise}.
\end{matrix}\right.
$$
Clearly, the extended adjacency matrix of a graph $G$ can be seen as a special weighted adjacency matrix of $G$,
which coincides exactly with the (ordinary) adjacency matrix of $G$ when $G$ is a regular graph.
If $G$ is an irregular graph, however, the extended adjacency matrix may contain more information of $G$.

In this paper, we are mainly concerned with the extended spectral radius of graphs.
Some lower and upper bounds on the extended spectral radius of graphs were established and the corresponding extremal graphs were characterized;
see \cite{Das,Ghorbani,wang} for details. Also, Nordhaus-Gaddum-type results for the extended spectral radius of graphs were given in \cite{wang}.
We here would like to consider the following problem, which is a variant of the famous Brualdi-Solheid's problem \cite{Brualdi}.

\begin{problem}\label{prob-1}
For a given class of graphs, characterize the graphs with the maximal or minimal extended spectral radius.
\end{problem}

Problem \ref{prob-1} is not easy to solve in general.
As a first step, we will settle this problem for the case of trees.
In the following, we let $\mathcal{T}_{n}$ be the set of all trees of order $n$,
and let $\eta_1(G)$ be the extended spectral radius of a graph $G$.

\begin{theorem}\label{T-1-2}
For any tree $T\in \mathcal{T}_n$ with $n\geq 5$, we have
$\eta_1(P_n)\leq\eta_1(T)\leq\eta_1(S_n)$,
with equality in the left (resp., right) inequality if and only if $T\cong P_n$ (resp., $T\cong S_n$).
\end{theorem}

Furthermore, we determine the first five trees with the maximal extended spectral radius.

\begin{theorem}\label{T-1-3}
For any tree $T\in\mathcal{T}_{n}\backslash\{S_n,T^1_n,T^2_n,T^3_n,T^4_n\}$ with $n\geq 12$, we have
$\eta_1(T)<\eta_1(T^4_n)<\eta_1(T^3_n)<\eta_1(T^2_n)<\eta_1(T^1_n)<\eta_1(S_n)$,
where $T^1_n$, $T^2_n$, $T^3_n$, and $T^4_n$ are the trees shown in Figure \ref{fig-1}.
\end{theorem}

The proofs of Theorems \ref{T-1-2} and \ref{T-1-3} will be given in Section 3.

\begin{figure}[ht]
\centering
\includegraphics[width=0.9\textwidth]{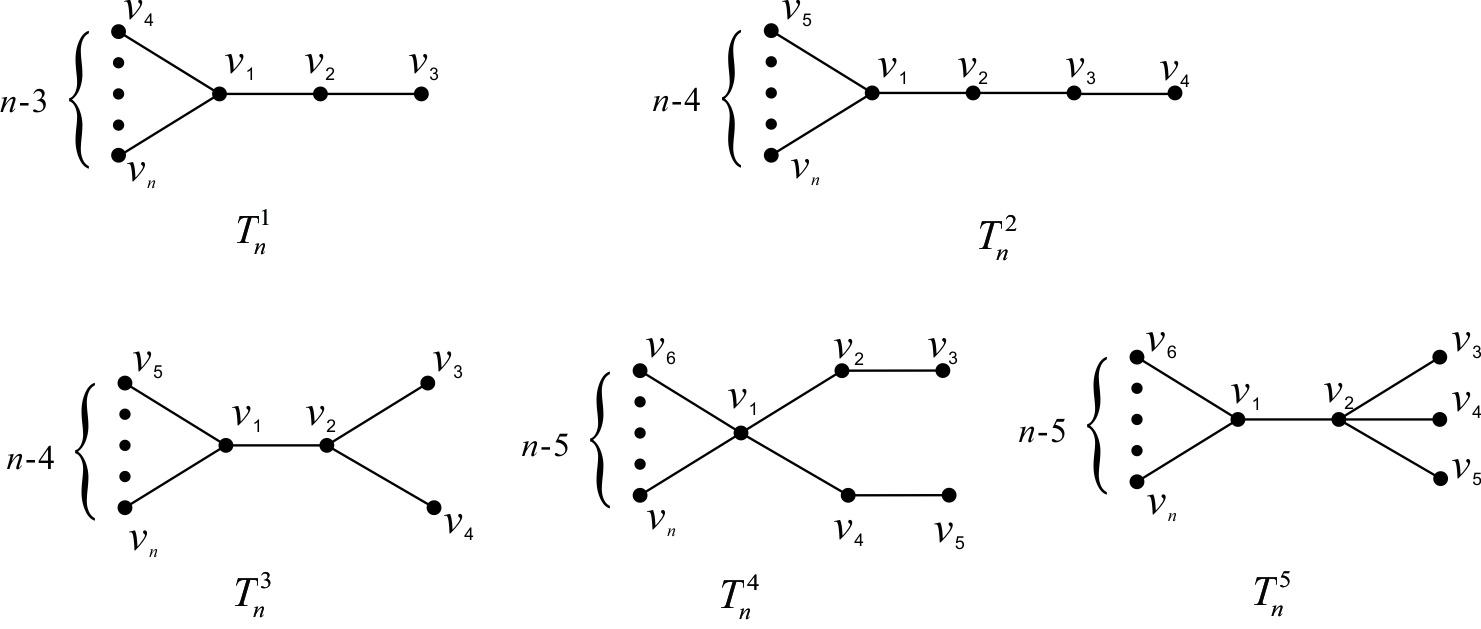}
\caption{The trees $T^i_n$ ($1\leq i\leq 5$).} \label{fig-1}
\end{figure}

\section{Preliminaries}

In this section, we shall present some results that will be needed to prove Theorems \ref{T-1-2} and \ref{T-1-3} in Section 3.

Let $\lambda_1(G)$ be the (ordinary) spectral radius of a graph $G$.
The following is a classical result concerning the maximal (ordinary) spectral radius of trees.

\begin{lemma}[\cite{A,Chang}] \label{Le-2-2}
Let $T\in \mathcal{T}_n$ with $n\geq 4$. Then
$$\lambda_1(T)\leq\sqrt{n-1},$$
and equality holds if and only if $T\cong S_n$.
Moreover, if $T\in\mathcal{T}_n\backslash\{S_n,T_n^1,T_n^2,T_n^3,T_n^4\}$ with $n\geq 11$, then
\begin{eqnarray*}
\lambda_1(T)\leq\sqrt{\frac{1}{2}\big(n-1+\sqrt{n^2-14n+61}\big)},
\end{eqnarray*}
and equality holds if and only if $T\cong T_{n}^{5}$ (depicted in Figure \ref{fig-1}).
\end{lemma}

The next result, due to Smith \cite{smith}, provides a complete characterization of connected graphs with $\lambda_1(G)\leq 2$.
\begin{lemma}[\cite{smith}]\label{Le-2-3}
Let $G$ be a connected graph on $n$ vertices with $n\geq 5$.
Then $\lambda_1(G)\leq2$ if and only if $G$ is one of the graphs:
$C_n, P_n, Z_n, W_n, H_1, H_2$, $H_3, H_4, H_5$ and $H_6$ (depicted in Figure \ref{fig-2}).
\end{lemma}

\begin{figure}[H]
\centering
\includegraphics[width=0.9\textwidth]{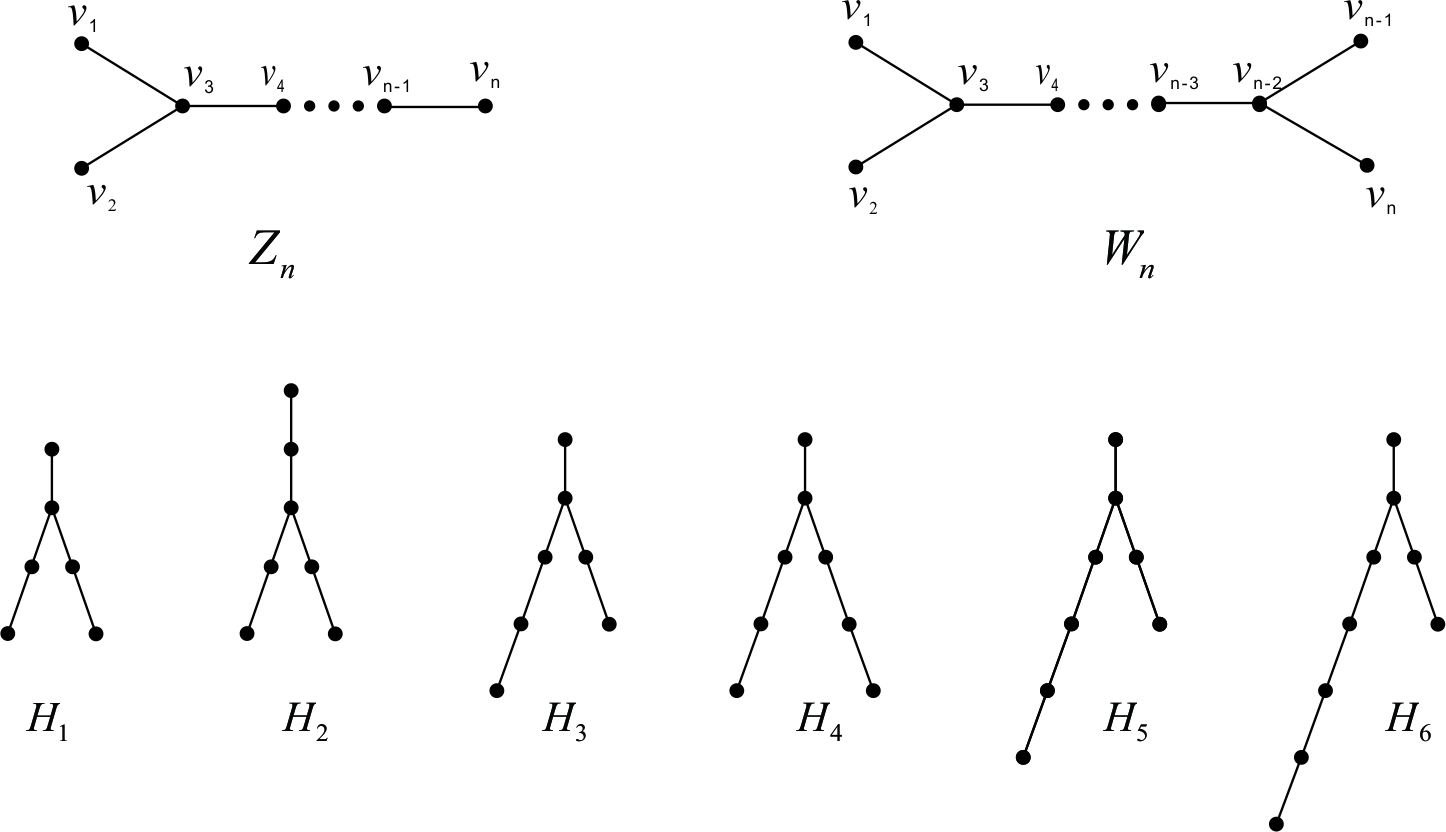}
\caption{The graphs $Z_n,W_n,H_1,H_2,H_3,H_4,H_5$, and $H_6$.} \label{fig-2}
\end{figure}

As observed in the introduction section, the extended adjacency matrix relates closely with the ordinary one,
so are their spectral radii, just as shown in the following result.

\begin{lemma}[\cite{Ghorbani}]\label{Le-2-4}
For any graph $G$ (with $\delta(G)\geq 1$), we have
\begin{eqnarray*}
\lambda_1(G)\leq\eta_1(G)\leq\frac{1}{2}\bigg(\frac{\Delta(G)}{\delta(G)}+\frac{\delta(G)}{\Delta(G)}\bigg)\lambda_1(G).
\end{eqnarray*}
The left equality holds if and only if $G$ is a regular graph and
the right equality holds if and only if $G$ is a regular graph or a bipartite semi-regular graph.
\footnote{The original statement on the right equality is that it holds if and only if $G$ is a complete bipartite graph, but this is not the case.}
\end{lemma}

For any graph $G$, let
$$M_1(G):=\sum_{v_i\in V(G)}d_i^2\quad \mbox{and} \quad F(G):=\sum_{v_i\in V(G)}d_i^3,$$
where $M_1(G)$ and $F(G)$ are named the first Zagreb index \cite{Gutman} and the forgotten index \cite{Furtula} of the graph $G$, respectively.
These two indices can be used to bound below the extended spectral radius of a graph.

\begin{lemma}[\cite{Ghorbani}]\label{Le-2-8}
For any graph $G$ (with $\delta(G)\geq 1$), we have
$\eta_1(G)\geq\frac{F(G)}{M_1(G)}$.
\end{lemma}

Let $\phi(B;x)$ 
be the characteristic polynomial of a (squared) matrix $B$.
If $H$ is an induced subgraph of a graph $G$, then $A_{ex}^{H}(G)$ will denote the principle submatrix of $A_{ex}(G)$,
which consists of the rows and columns corresponding to the vertices in $V(H)$.
Notice that $A_{ex}^H(G)$ may be different from $A_{ex}(H)$.

For two disjoint graphs $G_1$ and $G_2$ with $u\in V(G_1)$ and $v\in V(G_2)$,
we denote by $G_1 \cup G_2$ the disjoint union of $G_1$ and $G_2$, and by
$G_1(u,v)G_2$ the graph obtained from $G_1 \cup G_2$ by adding edge $uv$.
The next result can be used to simplify the calculation of $\phi(A_{ex}(G);x)$ in some cases (e.g., trees),
which follows from the same arguments as Lemmas 2.1 and 2.2 in \cite{Lin}; here we omit its proof.

\begin{lemma}\label{Le-2-11}

(i) If $G:=G_1 \cup G_2$, then
$$\phi(A_{ex}(G);x)=\phi(A_{ex}^{G_1}(G);x)\phi(A_{ex}^{G_2}(G);x).$$

(ii) If $G:=G_1(u,v)G_2$ and $f(d_u, d_v)=\frac{1}{2}(\frac{d_u}{d_v}+\frac{d_v}{d_u})$, then
\begin{eqnarray*}
\phi(A_{ex}(G);x)&=&\phi(A_{ex}^{G_1}(G);x)\phi(A_{ex}^{G_2}(G);x)\\
&&-f^2(d_u, d_v)\phi(A_{ex}^{G_1-u}(G);x)\phi(A_{ex}^{G_2-v}(G);x).
\end{eqnarray*}
\end{lemma}

\section{The proofs}

In this section, we shall present the proofs of Theorems \ref{T-1-2} and \ref{T-1-3}.

\vspace*{2mm}
\noindent
{\bf Proof of Theorem \ref{T-1-2}.}
We first prove the upper bound.
Let $T \in\mathcal{T} _{n}$ with $n\geq 5$.
Clearly, $\Delta(T)\leq n-1$ and $\delta(T)\geq 1$.
It is also easy to check that the function $x+\frac{1}{x}$ strictly increases with respect to $x$ when $x\geq 1$.
Thus, by Lemmas \ref{Le-2-4} and \ref{Le-2-2}, we have
$$\eta_1(T)\leq\frac{1}{2}\left(\frac{\Delta(T)}{\delta(T)}+\frac{\delta(T)}{\Delta(T)}\right)\lambda_1(T)\leq\frac{1}{2}\left(n-1+\frac{1}{n-1}\right)\sqrt{n-1}=\eta_1(S_n),$$
with equality if and only if $\Delta(T)=n-1$, $\delta(T)=1$, and $\lambda_1(T)=\sqrt{n-1}$, that is, $T\cong S_n$.

We next prove the lower bound.
Let $T\in\mathcal{T} _{n}\backslash\{P_n, Z_n, W_n, H_1, H_2, H_3, H_4,$ $H_5, H_6\}$.
By Lemmas \ref{Le-2-4} and \ref{Le-2-3}, we get
$\eta_1(T)\geq\lambda_1(T)>2.$
On the other hand, by Lemma \ref{Le-2-8}, we obtain
\begin{eqnarray*}
&&\eta_1(Z_n)\geq\frac{30+8(n-4)}{12+4(n-4)}>2,~
\eta_1(W_n)\geq\frac{58+8(n-6)}{22+4(n-6)}>2,\\
&&\eta_1(H_1)\geq\frac{46}{20}>2,~\eta_1(H_2)\geq\frac{54}{24}>2,~
\eta_1(H_3)\geq\frac{54}{24}>2,\\
&&\eta_1(H_4)\geq\frac{62}{28}>2,~
\eta_1(H_5)\geq\frac{62}{28}>2,~
\eta_1(H_6)\geq\frac{70}{32}>2.
\end{eqnarray*}
Now, in order to obtain the desired result, we just need to show that $\eta_1(P_n)<2$,
that is, $\phi(A_{ex}(P_n);x)>0$ when $x\geq 2$.

Indeed, it is known that (see e.g., \cite{Biggs}, p. 11)
\begin{eqnarray*}
\phi(A(P_n);x)=x\phi(A(P_{n-1});x)-\phi(A(P_{n-2});x),
\end{eqnarray*}
from which we can deduce that
\begin{eqnarray*}
\phi(A(P_{n});x)-\phi(A(P_{n-1});x)=(x-1)\phi(A(P_{n-1});x)-\phi(A(P_{n-2});x).
\end{eqnarray*}
Moreover, it is also known that $\lambda_1(P_n)=2\cos\frac{\pi}{n+1}<2$ (see e.g., \cite{Brouwer}, p. 9),
which yields that $\phi(A(P_{n});x)>0$ holds for $x\geq 2$. Thus, when $x\geq 2$, we get
\begin{eqnarray*}
\phi(A(P_{n});x)-\phi(A(P_{n-1});x) &\geq& \phi(A(P_{n-1});x)-\phi(A(P_{n-2});x)\\
&\geq& \phi(A(P_{n-2});x)-\phi(A(P_{n-3});x)\\
&\geq& \phi(A(P_{n-3});x)-\phi(A(P_{n-4});x)\\
&\geq& \cdots\\
&\geq& \phi(A(P_{2});x)-\phi(A(P_{1});x)\\
&=& x^2-x-1>0.
\end{eqnarray*}
Consequently, when $x\geq 2$, we obtain
\begin{eqnarray*}
&&\phi(A_{ex}(P_n);x)\\
&&~~~~=x^2\phi(A(P_{n-2});x)-2x\big(\frac{5}{4}\big)^2\phi(A(P_{n-3});x)+\big(\frac{5}{4}\big)^4\phi(A(P_{n-4});x)\\
&&~~~~\geq x^2\phi(A(P_{n-2});x)-\big[x^2+\big(\frac{5}{4}\big)^4\big]\phi(A(P_{n-3});x)+\big(\frac{5}{4}\big)^4\phi(A(P_{n-4});x)\\
&&~~~~=x^2\big[\phi(A(P_{n-2});x)-\phi(A(P_{n-3});x)\big]-\big(\frac{5}{4}\big)^4\big[\phi(A(P_{n-3});x)-\phi(A(P_{n-4});x)\big]\\
&&~~~~\geq \big[x^2-\big(\frac{5}{4}\big)^4\big]\big[\phi(A(P_{n-3});x)-\phi(A(P_{n-4});x)\big]>0,
\end{eqnarray*}
as required, completing the proof of Theorem \ref{T-1-2}.
\hfill $\square$

\vspace*{4mm}
\noindent
{\bf Proof of Theorem \ref{T-1-3}.}
For any tree $T$ with $n\geq 12$ and $\Delta(T)\leq n-4$,
by Lemmas \ref{Le-2-4} and \ref{Le-2-2} and the monotonicity of the function $x+\frac{1}{x}$, we obtain
\begin{eqnarray}
\eta_1(T)
&\leq&\frac{1}{2}\bigg(\frac{\Delta(T)}{\delta(T)}+\frac{\delta(T)}{\Delta(T)}\bigg)\lambda_1(T)\nonumber\\
&\leq&\frac{1}{2}\big(n-4+\frac{1}{n-4}\big)\sqrt{\frac{1}{2}\big(n-1+\sqrt{n^2-14n+61}\big)}\nonumber\\
&<&\frac{1}{2}(n-3)\sqrt{n-5}. \quad \textrm{(by Propositon \ref{prop-1} in Appendix)}\label{eq-1}
\end{eqnarray}

On the other hand, if $T$ is a tree with $n\geq 12$ and $\Delta(T)\geq n-3$,
then $T\in \{S_n, T_n^1, T_n^2, T_n^3,T_n^4\}$.
Furthermore, we have the following three claims.

\begin{claim} \label{clm-1}
If $n\geq 12$, then $\eta_1(T_n^1)>\eta_1(T_n^2)$.
\end{claim}

\begin{proof}
We first prove that $\eta_1(T_n^1)>\frac{1}{2}(n-2)\sqrt{n-3}$.
Label the vertices of $T_n^1$ as shown in Figure \ref{fig-1}.
Clearly, $A_{ex}(T_n^1)$ is a non-negative and irreducible matrix.
Thus, by the Perron-Frobenius theorem (see e.g., \cite{Brouwer}, p. 22), we know that $\eta_1(T_n^1)>0$
and there is a positive unit vector $\mathbf{x}=(x_1,x_2,x_3,x_4,\dots,x_n)^T$
such that $\eta_1(T_n^1) \mathbf{x}=A_{ex}(T_n^1)\mathbf{x}$.
This yields that $x_4=\cdots=x_n$ and
\begin{eqnarray*}
\left\{\begin{aligned}
\eta_1(T_n^1) x_1&=\frac{1}{2}\Big(\frac{2}{n-2}+\frac{n-2}{2}\Big) x_2+\frac{1}{2}(n-3)\Big(\frac{1}{n-2}+n-2\Big)x_4,\\
\eta_1(T_n^1) x_4&=\frac{1}{2}\Big(\frac{1}{n-2}+n-2\Big)x_1,
\end{aligned}\right .
\end{eqnarray*}
from which we can deduce that
\begin{eqnarray*}
\eta_1^2(T_n^1) x_1&>&\frac{1}{2}(n-3) \Big(\frac{1}{n-2}+n-2\Big)\eta_1(T_n^1) x_4\\
&=&(n-3)\bigg[\frac{1}{2}\Big(\frac{1}{n-2}+n-2\Big)\bigg]^2x_1,\\
&>&(n-3)\big[\frac{1}{2}(n-2)\big]^2x_1.
\end{eqnarray*}
Consequently, we have $\eta_1(T_n^1)>\frac{1}{2}(n-2)\sqrt{n-3}$, as desired.

Now, to obtain the required result, we just need to show that $\eta_1(T_n^2)<\frac{1}{2}(n-2)\sqrt{n-3}$.
Indeed, by Lemma \ref{Le-2-11} and some calculations, we get
$$\phi(A_{ex}(T_n^2);x)=x^{n-4}g_2(x),$$
where
\begin{eqnarray*}
g_2(x):=\bigg(x^2-\frac{(n-4)\big[(n-3)^2+1\big]^2}{4(n-3)^2}\bigg)\big(x^2-\frac{41}{16}\big)
-\frac{\big[(n-3)^2+4\big]^2}{16(n-3)^2}\big(x^2-\frac{25}{16}\big).
\end{eqnarray*}
For convenience, let $\eta:=\eta_1(T_n^2)$. Noting that $\eta^2-\frac{41}{16}>\frac{1}{2}(\eta^2-\frac{25}{16})$, we have
\begin{eqnarray*}
0=g_2(\eta)&=&\bigg(\eta^2-\frac{(n-4)\big[(n-3)^2+1\big]^2}{4(n-3)^2 }\bigg)\big(\eta^2-\frac{41}{16}\big)
-\frac{\big[(n-3)^2+4\big]^2}{16(n-3)^2}\big(\eta^2-\frac{25}{16}\big)\\
&>&\bigg(\eta^2-\frac{(n-4)\big[(n-3)^2+1\big]^2}{4(n-3)^2}-\frac{\big[(n-3)^2+4\big]^2}{8(n-3)^2}\bigg)\big(\eta^2-\frac{41}{16}\big)\\
&>&\bigg(\eta^2-\frac{(n-3)\big[(n-3)^2+1\big]^2}{4(n-3)^2}\bigg)\big(\eta^2-\frac{41}{16}\big).
\end{eqnarray*}
Consequently, since $\eta^2-\frac{41}{16}>0$ and $n\geq 12$, we get
\begin{eqnarray*}
\eta<\frac{1}{2}\bigg(n-3+\frac{1}{n-3}\bigg)\sqrt{n-3}<\frac{1}{2}(n-2)\sqrt{n-3},
\end{eqnarray*}
as desired, completing the proof of Claim \ref{clm-1}.
\end{proof}

\begin{claim}\label{clm-2}
If $n\geq 12$, then $\eta_1(T_n^2)>\eta_1(T_n^3)$.
\end{claim}

\begin{proof}
By applying Lemma \ref{Le-2-11} and some calculations, we have
$$\phi(A_{ex}(T_n^3);x)=x^{n-4}g_3(x),$$
where
\begin{eqnarray*}
g_3(x):=\bigg(x^2-\frac{(n-4)\big[(n-3)^2+1\big]^2}{4(n-3)^2}\bigg)\big(x^2-\frac{50}{9}\big)
-\frac{\big[(n-3)^2+9\big]^2x^2}{36(n-3)^2}.
\end{eqnarray*}
For convenience, let $\eta:=\eta_1(T_n^3)$.
By Lemma \ref{Le-2-8} and the condition $n\geq 12$, we have $\eta\geq\frac{(n-3)^3+3^3+(n-2)}{(n-3)^2+3^2+(n-2)}\geq 5$.
Moreover, from $g_3(\eta)=0$ it follows that
\begin{eqnarray*}
\eta^2-\frac{(n-4)\big[(n-3)^2+1\big]^2}{4(n-3)^2}
=\frac{\big[(n-3)^2+9\big]^2\eta^2}{36(n-3)^2\big(\eta^2-\frac{50}{9}\big)},
\end{eqnarray*}
and hence,
\begin{eqnarray*}
g_2(\eta)
&=&\frac{\big[(n-3)^2+9\big]^2\eta^2}{36(n-3)^2\big(\eta^2-\frac{50}{9}\big)}\big(\eta^2-\frac{41}{16}\big)
-\frac{\big[(n-3)^2+4\big]^2}{16(n-3)^2}\big(\eta^2-\frac{25}{16}\big)\\
&<&\bigg(\frac{\big[(n-3)^2+9\big]^2\eta^2}{36(n-3)^2\big(\eta^2-\frac{50}{9}\big)}
-\frac{\big[(n-3)^2+4\big]^2}{16(n-3)^2}\bigg)\big(\eta^2-\frac{41}{16}\big).
\end{eqnarray*}

Now, to get the desired result, we just need to show that $g_2(\eta)<0$.
Since $\eta^2-\frac{41}{16}>0$, this is equivalent to
$$\frac{\big[(n-3)^2+9\big]^2\eta^2}{36(n-3)^2\big(\eta^2-\frac{50}{9}\big)}
\leq\frac{\big[(n-3)^2+4\big]^2}{16(n-3)^2},$$
that is,
$$\eta^2\geq\frac{50\big[(n-3)^2+4\big]^2}{9\big[(n-3)^2+4\big]^2-4\big[(n-3)^2+9\big]^2}=\frac{10(n^2-6n+13)^2}{(n^2-6n+15)(n^2-6n+3)},$$
which always holds for $n\geq 12$ and $\eta\geq 5$, as required.
\end{proof}

\begin{claim}\label{clm-3}
If $n\geq 12$, then $\eta_1(T_n^3)>\eta_1(T_n^4)>\frac{1}{2}(n-3)\sqrt{n-5}$.
\end{claim}

\begin{proof}
We first prove that $\eta_1(T_n^3)>\eta_1(T_n^4)$.
As the proof of Claim \ref{clm-2}, by Lemma \ref{Le-2-11} and some calculations, we get
$$\phi(A_{ex}(T_n^4);x)=x^{n-6}g_4(x),$$
where
\begin{eqnarray*}
g_4(x):=\bigg(x^2-\frac{(n-5)\big[(n-3)^2+1\big]^2}{4(n-3)^2}\bigg)\big(x^2-\frac{25}{16}\big)^2
-\frac{\big[(n-3)^2+4\big]^2x^2}{8(n-3)^2}\big(x^2-\frac{25}{16}\big).
\end{eqnarray*}
For convenience, let $\eta:=\eta_1(T_n^4)$.
By Lemma \ref{Le-2-8} and the condition $n\geq 12$,
we have $\eta\geq\frac{(n-3)^3+2^3+2^3+(n-3)}{(n-3)^2+2^2+2^2+(n-3)}\geq 5$.
Moreover, from $g_4(\eta)=0$ it follows that
\begin{eqnarray*}
\eta^2-\frac{(n-4)\big[(n-3)^2+1\big]^2}{4(n-3)^2}
=\frac{\big[(n-3)^2+4\big]^2\eta^2}{8(n-3)^2(\eta^2-\frac{25}{16})}
-\frac{\big[(n-3)^2+1\big]^2}{4(n-3)^2},
\end{eqnarray*}
and hence,
\begin{eqnarray*}
g_3(\eta)&=&\bigg(\frac{\big[(n-3)^2+4\big]^2\eta^2}{8(n-3)^2(\eta^2-\frac{25}{16})}-\frac{\big[(n-3)^2+1\big]^2}{4(n-3)^2}\bigg)\big(\eta^2-\frac{50}{9}\big)
-\frac{\big[9+(n-3)^2\big]^2\eta^2}{36(n-3)^2}\\
&<&\bigg(\frac{\big[(n-3)^2+4\big]^2\eta^2}{8(n-3)^2(\eta^2-\frac{25}{16})}
-\frac{\big[(n-3)^2+1\big]^2}{4(n-3)^2}-\frac{\big[(n-3)^2+9\big]^2}{36(n-3)^2}\bigg)\big(\eta^2-\frac{50}{9}\big).
\end{eqnarray*}

Now, to get the desired result, we just need to show that $g_3(\eta)<0$.
Since $\eta^2-\frac{50}{9}>0$, this is equivalent to
\begin{eqnarray*}
\frac{\big[(n-3)^2+4\big]^2\eta^2}{8(n-3)^2(\eta^2-\frac{25}{16})}
\leq \frac{\big[(n-3)^2+1\big]^2}{4(n-3)^2}+\frac{\big[(n-3)^2+9\big]^2}{36(n-3)^2},
\end{eqnarray*}
that is,
\begin{eqnarray*}
\eta^2&\geq&\frac{25\big(9[(n-3)^2+1]^2+[(n-3)^2+9]^2\big)}{8\big(18[(n-3)^2+1]^2+2[(n-3)^2+9]^2-9[(n-3)^2+4]^2\big)}\\
&=&\frac{25( 5n^4 - 60n^3 + 288n^2 - 648n + 612)}{4(11n^4 - 132n^3 + 594n^2 - 1188n + 927)},
\end{eqnarray*}
which always holds for $n\geq 12$ and $\eta\geq 5$, as desired.

We next show that $\eta_1(T_n^4)>\frac{1}{2}(n-3)\sqrt{n-5}$.
Clearly, $S_{n-4}$ is a proper subgraph of $T_n^4$.
Furthermore, it is easy to check that
\begin{eqnarray*}
A_{ex}^{S_{n-4}}(T_n^4)=\frac{1}{2}\big(n-3+\frac{1}{n-3}\big)A(S_{n-4}),
\end{eqnarray*}
which is a principle submatrix of $A_{ex}(T_n^4)$.
Hence, by the interlacing theorem (see e.g., \cite{Brouwer}, p. 27) and Lemma \ref{Le-2-2}, we have
\begin{eqnarray*}
\eta_1(T_n^4)\geq\frac{1}{2}(n-3+\frac{1}{n-3})\sqrt{n-5}>\frac{1}{2}(n-3)\sqrt{n-5},
\end{eqnarray*}
as required, completing the proof of Claim \ref{clm-3}.
\end{proof}

Now, by combining the above three claims, as well as Theorem \ref{T-1-2} and (\ref{eq-1}), we obtain the desired result,
completing the proof of Theorem \ref{T-1-3}.
\hfill $\square$

\section{Further discussions}

In the previous section, we have shown that the path $P_n$ is the unique $n$-vertex tree with the minimal extended spectral radius.
Here we can say more about it. Indeed, from the proof of Theorem \ref{T-1-2},
we know that $\eta_1(P_n)<2$, $\eta_1(Z_n)>2$, $\eta_1(W_n)>2$, and $\eta_1(H_i)>2$ for $i=1,2,\ldots,6$.
Furthermore, observing that $C_n$ is a regular graph of degree 2, we have $\eta_1(C_n)=\lambda_1(C_n)=2$.
Also, for any connected graph $G\notin\{C_n,P_n,Z_n,W_n,H_1,H_2,H_3,H_4,H_5,H_6\}$,
Lemmas \ref{Le-2-4} and \ref{Le-2-3} yield that $\eta_1(G)\geq \lambda_1(G)>2$.
Now, by combining these arguments, we can determine the connected graphs having the minimal and the second-minimal extended spectral radius.

\begin{theorem}\label{T-4-1}
Let $G$ be a connected graph on $n$ vertices with $n\geq 5$. Then
$\eta_1(G)\geq \eta_1(P_n)$, and equality holds if and only if $G\cong P_n$.
If $G\ncong P_n$, then $\eta_1(G)\geq$ $\eta_1(C_n)=2$, and equality holds if and only if $G\cong C_n$.
\end{theorem}

By Theorem \ref{T-4-1}, and noting that $C_n$ is a unicycle graph and that $P_n$ is a bipartite graph,
we can obtain the following two simple corollaries.

\begin{corollary}
Let $G$ be a unicycle graph on $n$ vertices with $n\geq 5$. Then
$\eta_1(G)\geq \eta_1(C_n)=2$, and equality holds if and only if $G\cong C_n$.
\end{corollary}

\begin{corollary}
Let $G$ be a bipartite connected graph on $n$ vertices with $n\geq 5$. Then
$\eta_1(G)\geq \eta_1(P_n)$, and equality holds if and only if $G\cong P_n$.
\end{corollary}

By contrast, however, characterization of the (bipartite) connected graphs with the maximal extended spectral radius seems hard,
since for a (bipartite) connected graph $G$,
addition of an edge may lead to increase or decrease of the extended spectral radius of $G$ (see \cite{Ghorbani} for some comments),
which is very different from the (ordinary) spectral radius of $G$.
We here pose the following conjecture based on numerical experiments.

\begin{conj} \label{conj-1}
If $G$ is a (bipartite) connected graph of order $n\geq 5$, then
$$\eta_1(G)\leq \eta_1(S_n)=\frac{n^2-n+2}{2\sqrt{n-1}},$$
and equality holds if and only if $G\cong S_n$.
\end{conj}

We provide two more evidences for Conjecture \ref{conj-1}.
First, if $G$ is a regular graph of degree $k$, then
$\eta_1(G)=\lambda_1(G)=k\leq n-1<\frac{n^2-n+2}{2\sqrt{n-1}}$.
Second, if $G$ is a complete bipartite graph $K_{a,\,b}$ ($a+b=n$ and $a\geq b\geq 1$), then
$$\eta_1(K_{a,\,b})=\frac{a^2+b^2}{2\sqrt{ab}}\leq\frac{n^2-n+2}{2\sqrt{n-1}}=\eta(S_n),$$
and equality holds if and only if $a=n-1$ and $b=1$, i.e., $K_{a,\,b}\cong S_n$.

Finally, it should be mentioned that, as an extension of the (ordinary) adjacency matrix of a graph,
the extended adjacency matrix contains more information about the graph (except the regular one),
which can be used to explain mathematically why the extended spectral radius and the extended energy possess high discriminating power
and correlate well with a number of physicochemical properties and biological activities of organic compounds.
However, the extension will lead to more complicated treatment on the study of the extended adjacency matrix.

\section*{Acknowledgments}
This work was supported by the National Science Foundation of China under Grant No. 11861011.

\section*{Appendix}

\begin{prop}\label{prop-1}
If $n\geq 12$, then $$\frac{1}{2}(n-3)\sqrt{n-5}>\frac{1}{2}\big(n-4+\frac{1}{n-4}\big)\sqrt{\frac{1}{2}\big(n-1+\sqrt{n^2-14n+61}\big)}.$$
\end{prop}

\begin{proof}
It is easy to check that
\begin{eqnarray} \label{eq-2}
&&\frac{1}{2}(n-3)\sqrt{n-5}>\frac{1}{2}\big(n-4+\frac{1}{n-4}\big)\sqrt{\frac{1}{2}\big(n-1+\sqrt{n^2-14n+61}\big)}\nonumber\\
\Leftrightarrow &&\big[2(n-3)^2(n-4)^2(n-5)-\big((n-4)^2+1\big)^2(n-1)\big]^2\nonumber\\
&&-\big((n-4)^2+1\big)^4(n^2-14n+61)>0\nonumber\\
\Leftrightarrow &&4(n - 5)(n^8-39n^7+641n^6-5882n^5+33246n^4-119207n^3\nonumber\\
&&+265885n^2-338328n+188499)>0.
\end{eqnarray}

In order to prove (\ref{eq-2}), we consider the function
$f(x):=x^8-39x^7+641x^6-5882x^5+33246x^4-119207x^3+265885x^2-338328x+188499$ with $x\geq 12$,
whose derivatives are given by
\begin{eqnarray*}
&&f^{(8)}(x)=40320,\, f^{(7)}(x)=40320x - 196560,\\
&&f^{(6)}(x)=20160x^2 - 196560x + 461520,\\
&&f^{(5)}(x)=6720x^3 - 98280x^2 + 461520x - 705840,\\
&&f^{(4)}(x)=1680x^4 - 32760x^3 + 230760x^2 - 705840x + 797904,\\
&&f^{\prime\prime\prime}(x)=336x^5 - 8190x^4 + 76920x^3 - 352920x^2 + 797904x - 715242,\\
&&f^{\prime\prime}(x)=56x^6 - 1638x^5 + 19230x^4 - 117640x^3 + 398952x^2 - 715242x \\
&&~~~~~~~~~~ + 531770,\\
&&f^{\prime}(x)=8x^7 - 273x^6 + 3846x^5 - 29410x^4 + 132984x^3 - 357621x^2 \\
&&~~~~~~~~~~ + 531770x - 338328 .
\end{eqnarray*}
When $x\geq 12$, since $f^{(8)}(x)>0$, we know that the function $f^{(7)}(x)$ increases strictly with respective to $x$,
and hence $f^{(7)}(x)\geq f^{(7)}(12)=287280>0$.
Similarly, when $x\geq 12$, we have
\begin{eqnarray*}
&&~~~~~~f^{(7)}(x)>0\\
&&\Longrightarrow f^{(6)}(x)\geq f^{(6)}(12)=1005840>0
\Longrightarrow
f^{(5)}(x)\geq f^{(5)}(12)=2292240>0\\
&&\Longrightarrow
f^{(4)}(x)\geq f^{(4)}(12)=3784464>0
\Longrightarrow f^{\prime\prime\prime}(x)\geq f^{\prime\prime\prime}(12)=4736598>0\\
&&\Longrightarrow f^{\prime\prime}(x)\geq f^{\prime\prime}(12)=4497602>0
\Longrightarrow f^{\prime}(x)\geq f^{\prime}(12)=2984784>0\\
&&\Longrightarrow f(x)\geq f(12)=742467>0,
\end{eqnarray*}
as desired. This completes the proof of Proposition \ref{prop-1}.
\end{proof}

\end{document}